\UseRawInputEncoding
\documentclass[12pt]{amsart}

\usepackage{latexsym, amssymb, amsmath,ulem,soul,esint}

\usepackage{amsfonts, graphicx}
\usepackage{graphicx,color}
\newcommand{\be}{\begin{equation}}
\newcommand{\ee}{\end{equation}}
\newcommand{\beq}{\begin{eqnarray}}
\newcommand{\eeq}{\end{eqnarray}}

\newtheorem{thm}{Theorem}[section]

\newtheorem{lma}{Lemma}[section]
\newtheorem{prop}{Proposition}[section]
\newtheorem{cor}{Corollary}[section]

\theoremstyle{remark}
\newtheorem{rem}{Remark}[section]
\numberwithin{equation}{section}
\def\tr{\operatorname{tr}}
\def\dps{\displaystyle}

\def\be{\begin{equation}}
\def\ee{\end{equation}}
\def\bee{\begin{equation*}}
\def\eee{\end{equation*}}

\def\lf{\left}
\def\ri{\right}

\def\Ric{\text{\rm Ric}}
\def\Rm{\text{\rm Rm}}

\def\wt{\widetilde}
\def\la{\langle}
\def\ra{\rangle}
\def\p{\partial}

\def\heat{\lf(\frac{\p}{\p t}-\Delta_{g(t)}\ri)}
\def\tr{\operatorname{tr}}

\def\e{\varepsilon}
\def\a{{\alpha}}
\def\b{{\beta}}

\def\R{\mathbb{R}}

\def\n{\nabla}
\def\ppt{\frac{\partial}{\partial t}}

\def\Ric{\text{\rm Ric}}
\def\Rm{\text{\rm Rm}}

\def\wt{\widetilde}
\def\la{\langle}
\def\ra{\rangle}
\def\p{\partial}

\def\tr{\operatorname{tr}}

\def\e{\varepsilon}
\def\a{{\alpha}}
\def\b{{\beta}}

\def\R{\mathbb{R}}

\def\n{\nabla}
\def\ppt{\frac{\partial}{\partial t}}

\def\Ric{\text{\rm Ric}}
\def\Rm{\text{\rm Rm}}

\def\wt{\widetilde}
\def\la{\langle}
\def\ra{\rangle}
\def\p{\partial}

\def\tr{\operatorname{tr}}

\def\e{\varepsilon}
\def\a{{\alpha}}
\def\b{{\beta}}

\def\R{\mathbb{R}}

\def\n{\nabla}
\def\ppt{\frac{\partial}{\partial t}}

\def\n{\nabla}
\def\tn{\wt\nabla}

\def\ppt{\frac{\partial}{\partial t}}

\def\b{\beta}
\def\t{\tilde}

\begin{document}

\title{Short time existence for harmonic map heat flow with time-dependent metrics}

 \author{Shaochuang Huang$^1$}
\address[Shaochuang Huang]{Department of Mathematics, Southern University of Science and Technology, Shenzhen, Guangdong, China.}
\email{huangsc@sustech.edu.cn}
\thanks{$^1$Research   partially supported by NSFC \#1200011128}

\author{Luen-Fai Tam$^2$}
\address[Luen-Fai Tam]{The Institute of Mathematical Sciences and Department of
 Mathematics, The Chinese University of Hong Kong, Shatin, Hong Kong, China.}
 \email{lftam@math.cuhk.edu.hk}
 \thanks{$^2$Research partially supported by Hong Kong RGC General Research Fund \#CUHK 14301517}

\renewcommand{\subjclassname}{
  \textup{2020} Mathematics Subject Classification}

\begin{abstract}
In this work, we  obtain  a short time existence result for harmonic map heat flow coupled with a smooth family of complete metrics in the domain manifold. Our results generalize short time existence results for harmonic map heat flow by Li-Tam \cite{Li-Tam} and Chen-Zhu \cite{Chen-Zhu}. In particular, we prove the short time existence of harmonic map heat flow along a complete Ricci flow $g(t)$ on $M$ into a complete manifold with curvature bounded from above with a smooth initial map of uniformly bounded energy density, under the assumptions that $|\Rm(g(t))|\le a/t$ and $g(t)$ is uniformly equivalent to $g(0)$.
\end{abstract}


\maketitle

\markboth{Shaochuang Huang and Luen-Fai Tam }{Short time existence for harmonic map heat flow}

\section{Introduction}

In this work, we want to extend some previous short time existence results of harmonic map heat flow. Harmonic map heat flow  was first introduced by Eells and Sampson \cite{E-S} to obtain harmonic map between two Riemannian manifolds. As a first step they proved the short time existence for harmonic map heat flows between compact manifolds. Later in \cite{Li-Tam}, Peter Li and the second author proved the short time existence from a complete noncompact manifold $(M^m,g)$ to another complete Riemannian manifold $(N^n,h)$ with Ricci curvature of $g$ satisfying $\Ric(g)\ge -Kg$ for some $K\ge0$ and initial map $f$ with bounded energy density so that $f(M)$ is {\it bounded}. Under an additional condition that the curvature $\Rm(h)$ of $h$ is nonpositive, one can remove the assumption that $f(M)$ is bounded. From the point of view of PDE, one would like to understand whether one  can still obtain short time solution by only assuming that $\Rm(h)\le \kappa$ for some $\kappa\ge0$ without assuming that $f(M)$ is bounded.

On the other hand, Hamilton \cite{Hamilton1995} used harmonic map heat flow on compact manifolds along a Ricci flow of the domain manifold to obtain uniqueness result of Ricci flow. Later, Chen and Zhu  studied the  uniqueness of Ricci flow on non-compact manifold  following  Hamilton's approach. In \cite{Chen-Zhu}, Chen and Zhu proved that if  a  Ricci flow $g(t)$, $0\le t\le T$, which is complete on a noncompact manifold $M$ has {\it uniformly bounded curvature}, then one can obtain short time solution for harmonic map heat flow along the Ricci flow from $(M,g(t))$    to $(M,g(T))$   with identity map as initial data. From this together with some careful estimates, they obtained uniqueness result on Ricci flow with uniformly bounded curvature on noncompact manifold.

The uniqueness result was generalized to Ricci flow which may have unbounded curvature. In \cite{Kotschwar}, Kotschwar introduced an energy method to obtain a more general uniqueness result. The method has been developed further   by   Lee  \cite{Lee} and Ma-Lee \cite{ Ma-Lee}. In \cite{Ma-Lee}, Ma and Lee proved that if two complete solutions of the Ricci flows with the same initial metric on  a noncompact manifolds with curvature bounded by $a/t$ for some $a>0$ so that the deformed metrics are uniformly equivalent to the initial metric, then they are the same.

One may wonder if one can use harmonic map heat flow to obtain similar results. Short time existence results on harmonic map heat flow  in \cite{Li-Tam} and the above uniqueness results on Ricci flow motivate the study of this work.

Our main result can be described as follows.   Let $M^m$ be a noncompact manifold and let $g(t)$ be a smooth family of complete metrics defined on $M\times[0,T]$ so that
\be\label{e-g-flow}
\frac{\p }{\p t}g(x,t)=H(x,t).
\ee
Let $ (N^n,h)$  be another complete Riemannian manifold. We want to study the initial value problem for the harmonic map heat flow:
\be\label{e-harmonicheat-1}
\left\{
  \begin{array}{ll}
 \frac{\p}{\p t}F(x,t)=\tau(F)(x,t)\\
F(x,0)=f(x)
  \end{array}
\right.
\ee
where $f(x)$ is a smooth map from $M$ to $N$ and $\tau(F)(x,t)$ is the tension field of the map $F(\cdot,t):M\to N$ with respect to $g(t)$ and $ h$. For more details of the definitions of harmonic map heat flow and related quantities, see \S\ref{s-setup}.

Consider the following assumptions:

\vskip .3cm

\begin{itemize}
  \item[{\bf (a1)}] $2\Ric(g(t))+H(t)\ge -K(t)g(t)$ in $M\times [0,T]$ where $K(t)\ge0$ and
  $$K_0=:\int_0^TK(t)dt<\infty.
      $$
  \item[{\bf (a2)}] $|H|\le at^{-1}$  and $|\nabla H|\le at^{-\frac32}$ for some $a>0$. Here the norm and the covariant derivative are with respect to $g(t)$.
  \item[{\bf (a3)}] The curvature of $h$ is bounded from above: $\Rm(h)\le \kappa$ for some $\kappa\ge0$.
\end{itemize}

\vskip .3cm

We obtain the following short time existence result.

\begin{thm}\label{t-existence-main} Let $(M^m, g(t))$ and $(N,h)$ be as above satisfying assumptions {\bf(a1)--(a3)}. Suppose there exists a smooth exhaustion function $\gamma$ on $M$ and $C_0>0$ such that$$
       d_{T}(p,x)+1\leq\gamma(x)\leq  d_{T}(p,x)+C_0$$ and $$
|\nabla^k_T \gamma| \le C_0
$$for $1\le k\le 2$, where $d_{T}$ is the distance function and $\nabla_T$ is the covariant derivative   with respect to $g(T)$. Given any smooth map $f:M\to N$ such that
$$
\sup_M e(f;g(t))\le e_0
$$
for all $t\in [0,T]$ for some constant $e_0$ where $e(f;g(t))$ is the energy density of the map $f$ from $(M,g(t))$ to $(N,h)$,    the harmonic heat flow \eqref{e-harmonicheat-1} has a short time  smooth solution $F$ with initial map $f$ defined on $M\times[0,T_0]$ such that
$$
\sup_{M\times [0,T_0]}e(F)\le C ; \sup_{M\times[0,T_0]}|\tau(F)|_{g(t)}\le C t^{-\frac12}
$$
for some $C>0$ depending only on $m, n, K_0, \kappa, a, e_0$ and

$$
T_0=\min\{T, \frac12\lf( 2\kappa e_0 \exp(K_0)\ri)^{-1}\}.
$$
In particular, if $\kappa=0$, then the harmonic map heat flow exists on $M\times[0,T]$.
\end{thm}

In the theorem, the assumption on the existence of $\gamma$ is satisfied if $g(T)$ has bounded curvature, see \cite{Tam}. The condition that $e(f;g(t))$ is uniformly bounded is satisfied if (i) $e(f;g(0))$ is uniformly bounded and if $g(t)$ is uniformly equivalent to $g(0)$;  or more generally  (ii) $e(f;g(t_0))$ is uniformly bounded for some $t_0$ and $g(t)\ge Cg(t_0)$ for some $C>0$. We should also remark that the bounds in conclusion of the theorem do not depend on $C_0$.

  Suppose $g(t)=g(0)$ is fixed, then $H=0$.  Then {\bf(a1)} is satisfied if $\Ric(g)\ge -Kg$ for some $K\ge0$. If $g$ has bounded curvature, then the short time existence result in \cite{Li-Tam} is still true without assuming the initial map has bounded image, provided $(N,h)$ has curvature bounded from above. See Corollary \ref{c-LT}.

If $g(t)$ is a solution to the Ricci flow, we have:
\begin{thm}\label{t-Ricci}
Let $(M^m,g(t))$, $t\in [0,T]$ with $T>0$, be a complete solution of the Ricci flow on a noncompact manifold. Suppose $(N^n,h)$ is another complete manifold with $\Rm(h)\le \kappa$ for some $\kappa\geq 0$. Let $f:M\to N$ be a smooth map with bounded energy density, namely, $\sup_M e(f;g(0))\le e_0$. Assume that $|\Rm(g(t))|\le a/t$ for some $a>0$ on $M\times[0,T]$ and assume that $g(t)\ge bg(0)$ for some $b>0$ on $M\times[0,T]$, then there exists a smooth solution $F$ to the heat flow for harmonic map along   $g(t)$ with initial map $f$ defined on $M\times[0,T_0]$ such that
$$
\sup_{M\times [0,T_0]}e(F)\le C_1; |\tau(F)|_{g(t)}(\cdot, t)\le C_1t^{-\frac12}
$$
for some $ C_1>0$ depending only on $m, n, \kappa, a, b, e_0$ and
$$
T_0=\min\{T, \frac12\lf( 2\kappa be_0 \ri)^{-1}\}.
$$
In particular, if $\kappa=0$ then the harmonic map heat flow exists on $M\times[0,T]$.
\end{thm}
 The theorem is a corollary of Theorem \ref{t-existence-main} by the fact that $H=-2\Ric$ in this case and by the covariant derivatives estimates of the curvature tensor along Ricci flow by Shi \cite{Shi}.

To prove our results, instead of solving Dirichlet problem as in \cite{Chen-Zhu}, we will use the method of iteration which was introduced by Eells and  Sampson in their seminal  work \cite{E-S} and was also used   in \cite{Li-Tam}.  One of the key point is to obtain good estimates for the fundamental solution of the heat equation.   For the case of fixed metric,  the estimates are contained in \cite{Li-Yau}. For the case of time-dependent metrics, we apply the    estimates in \cite{Chau-Tam-Yu} instead. We also obtain a new  estimate, see Theorem \ref{int-G}.

Finally, we would like to point out that we  are still unable to give another proof of the uniqueness result on Ricci flow as in \cite{Ma-Lee}. The main difficulty is that the second fundamental form of the identity map from $(M,g(0))$ to $(M,g(T))$ may not be bounded. If the curvature of the Ricci flow $|\Rm(g(t))|\le at^{-1+\a}$ with $\a>\frac12$, then one can prove that the second fundamental form mentioned above is bounded and one can obtain uniqueness. On the other hand, hopefully the results in Theorem \ref{t-existence-main} may have other applications.

 This paper is organized as follows. In Section \ref{pre}, we give estimates for the  fundamental solution   of heat operator and give a proof of a generalized maximum principle. In Section \ref{s-basicheat}, we study linear  heat equations for homogeneous and in-homogeneous cases and a  semi-linear heat equation closely related to the harmonic map heat flow. In Section \ref{s-harmonic-heat}, we study  the harmonic map heat flow and give a proof of Theorem \ref{t-existence-main}.

\section{Preliminary}\label{pre}

In this section, we will describe some     estimates of   the  fundamental solution (Green's function)  of heat operator with time-dependent complete  metrics on a noncompact manifold, which will be used later. We will also extend a maximum principle.

\subsection{Estimates for Fundamental Solution}\label{S-Green}

Let $ g(t), t\in[0, T]$ be a    family of complete Riemannian metrics on a manifold $M^m$. We always assume $M$ is noncompact and $g(t)$ is smooth in space and time. Recall that   $G$ is the fundamental solution of heat operator $\ppt-\Delta_{g(t)}$ if it satisfies
\be\label{e-kernel-1}
\left\{
  \begin{array}{ll}
  (\partial_t -\Delta_{x,t})\, G (x,t;y,s)=0, & \text{in $ M \times M\times(s,T]$}; \\
     \lim_{t\rightarrow s^+}G(x,t;y,s) =\delta_{y}(x), & \text{for $y\in M$}.
  \end{array}
\right.
\ee

Let $$H:=\ppt g.$$
Suppose $|H(x,t)|, |\n H|(x,t)$ and $|\Rm (g)|(x,t)$ are uniformly bounded in space and time, where the norms and covariant derivatives are taken with respect to $g(t)$. It is known that the fundamental solution exists and is positive, see \cite{Guenther} for example. We have the following estimates for $G$, see \cite{Chau-Tam-Yu}.

\begin{thm}\label{int-G} Let $  g(t) , t\in[0, T]$ be a family of smooth complete metrics on $M$ as above  with $|H|_g\leq H_0$, $|\n H|_g\leq H_1$ and $|\Rm(g(t))|\leq k_0$.  Then we have the following:
\begin{enumerate}
  \item [{\bf(a)}] \cite[Theorem 5.5]{Chau-Tam-Yu} There are   constants $C, D>0$ depending only on $H_0, k_0, m, T$ such that
  \bee
  G(x,t;y,s)\le \frac{C}{V^{\frac 12}_x(\sqrt{t-s})V^{\frac 12}_y(\sqrt{t-s})}\exp\lf(-\frac{r^2(x,y)}{D(t-s)}\ri)
  \eee
  for any $0\leq s<t\leq T$. Here $r(x,y)$ is the distance and $V_x(\rho)$ is the volume of the geodesic ball of radius $\rho$ with center at $x$ with respect to $g(0)$.
  \item [{\bf(b)}]\cite[Corollary 4.4]{Chau-Tam-Yu} Fix $\a>1$. For any $\delta>0$, we have\bee
  G(p, t; y, s)\leq (1+\delta)^{m\a/2}\cdot\exp{(A\delta t+\frac{B\a}{\delta t}r^2_t(p, q))}\cdot  G(q, (1+\delta)t; y, s), \eee
 where $A>0$ depends only on $m, T, H_0, H_1, k_0, \a$ and $B$ depends only on $H_0, T$.   \item[{\bf(c)}] For any $x\in M$,  $0\leq s<t\leq T$, we have
 \bee
 \int_M G(x, t; y, s)dV_s(y)=1.
 \eee
\end{enumerate}
\end{thm}
\begin{proof} (a) and (b) are from \cite{Chau-Tam-Yu}. It remains to prove (c).
We use similar idea in the proof of Lemma 5.1 in \cite{Chau-Tam-Yu}. Because the curvature of $g(0)$ is bounded. We can find a smooth function $\rho$ so that
$$
C^{-1}(r(x)+1)\le \rho(x)\le C(r(x)+1), |\nabla_{g(0)}\rho|+|\nabla^2_{g(0)}\rho|\le C
$$
for some $C>0$ depending only on $k_0$ and $m$. Here $r(x)$ is the distance from a fixed point $p$ with respect to $g(0)$. Let  $\eta$ be a smooth cut-off function such that $0\leq\eta\leq 1$, $\eta=1$ on $[0, 1]$ and $\eta=0$ on $[2,+\infty)$, $\eta>0$ on $[0, 2)$, $0\geq \eta'/\eta^{\frac{1}{2}}\geq-C_0$ and $\eta''\geq-C_0$ on $[0, +\infty)$ with $C_0$ being a positive absolutely constant. Let $\phi=\eta(\rho/R)$.

For $0\leq s_1<s_2<t$, we have
\bee\begin{split}
&\lf|\int_M\phi G(x,t; y, s_2)dV_{s_2}(y)-\int_M\phi G(x,t;y, s_1)dV_{s_1}(y)\ri|\\
=&\lf|\int^{s_2}_{s_1}(\frac{\partial}{\partial s}\int_M\phi G(x,t;y,s)dV_{s}(y))ds\ri|\\
=&\lf|\int^{s_2}_{s_1}\int_M\phi \lf(\frac{\partial}{\partial s}G(x,t;y,s)+h(y)G(x,t;y,s)\ri)dV_{s}(y)ds\ri|\\
=&\lf|\int^{s_2}_{s_1}\int_M\phi \Delta_{s,y}G(x,t;y,s)dV_{s}(y)ds\ri|\\
=&\lf|\int^{s_2}_{s_1}\int_M G(x,t;y,s) \Delta_{s,y}\phi dV_{s}(y)ds\ri|,\\
\end{split}\eee
where $h=\frac12\tr_{g}H$, and  we have used the fact that $G$ is also the fundamental solution of the conjugate heat equation i.e. $(-\frac{\partial}{\partial s}-\Delta_{s,y}-h(y))G=0$.
Now
\bee
 \Delta_{s,y}\phi =\frac1R\phi'\Delta_{s,y}\rho+\frac1{R^2}\phi''|\nabla_{g(s)} \rho|^2.
\eee
 Since $|H|, |\nabla H|$ are uniformly bounded, we conclude that
 $$
 |\Delta_{s,y}\phi|\le C
 $$
 for some constant independent of $R$ and $s$. This implies that
 \bee
 \lf|\int_M\phi G(x,t; y, s_2)dV_{s_2}(y)-\int_M\phi G(x,t;y, s_1)dV_{s_1}(y)\ri|\le \frac{C'}R
 \eee
 for some constant $C'$ independent of $s_1,s_2$, where we have  also used \cite[Corollary 5.2]{Chau-Tam-Yu} so that
 $$
 \int_M\phi G(x,t; y, s )dV_{s }(y)\le c
 $$
 for some constant $c$ independent of $s$. Let $R\to \infty$   and note that $$\lim\limits_{s\to t^{-}}\int_M G(x,t;y,s)dV_s(y)=1,
 $$ we obtain  \bee
 \int_M G(x, t; y, s)dV_s(y)=1. \eee
 The result follows.
\end{proof}
By the theorem, we can proceed as in the proof of \cite[Lemma 2.1]{Li-Tam} to have the following:
\begin{cor}\label{c-G}
\be\label{difference-kernel}
\int_M\lf|G(p, t; y, s)-G(q, t; y, s)\ri| dV_s(y) \leq C\cdot\frac{r_t(p, q)}{\sqrt{t-s}} \ee for any $p, q\in M$ and $0\leq s<t\leq T$. Here $C$ is a constant depending only on $m, H_0, H_1, k_0$ and $T$. Here $r_t$ is the distance function with respect to $g(t)$.
\end{cor}
\begin{proof} The proof is exactly as in \cite{Li-Tam}. We sketch the argument here for the sake of completeness. Let $\delta>0$ and $1<\a<4$ to be determined later.
\bee\begin{split}
  &\int_M |G(p, t; y, s) -G(q, t; y, s)|dV_s(y)\\
\leq& \int_M|G(q, (1+\delta)t; y, s)-G(q, t; y, s)|dV_s(y)\\
& + \int_M|G(p, t; y, s) -G(q, (1+\delta)t; y, s)|dV_s(y)\\
=& \mathrm{(I)+(II)}. \end{split}\eee
By Theorem \ref{int-G}(b) and (c):

 \bee\begin{split}
 \mathrm{(I)}\leq&  \int_M|(1+\delta)^{m\a/2}\cdot\exp{(A\delta t)}\cdot  G(q, (1+\delta)t; y, s)-G(q, t; y, s)|dV_s(y)\\
 &+\int_M|(1+\delta)^{m\a/2}\cdot\exp{(A\delta t)}\cdot  G(q, (1+\delta)t; y, s)- G(q, (1+\delta)t; y, s)|dV_s(y)\\
 \leq&  \int_M(1+\delta)^{m\a/2}\cdot\exp{(A\delta t)}\cdot  G(q, (1+\delta)t; y, s)-G(q, t; y, s)dV_s(y)\\
 &+\int_M[(1+\delta)^{m\a/2}\cdot\exp{(A\delta t)}-1]\cdot  G(q, (1+\delta)t; y, s)dV_s(y)\\
 =& 2[(1+\delta)^{m\a/2}\cdot\exp{(A\delta t)}-1]. \end{split}\eee
Here $A$ is a constant in the theorem. By Theorem \ref{int-G}(b) and (c) again,
 \bee
 \mathrm{(II)}\leq  2[(1+\delta)^{m\a/2}\cdot\exp{(A\delta t+\frac{B\a}{\delta t}r^2_t(p, q))}-1]. \eee
Here $B$ is also the constant in the theorem. Here $A$ and $B$ are independent of $\delta$.
 Hence \bee\begin{split}
  \int_M|G(p, t; y, s)-G(q, t; y, s)|dV_s(y)
\leq  4[(1+\delta)^{m\a/2}\cdot\exp{(A\delta t+\frac{B\a}{\delta t}r^2_t(p, q))}-1].\end{split}\eee
Let $r=r_t(p,q)$. If $\frac{r}{\sqrt{t}}> \frac{1}{2\sqrt{B}}$,  we have \bee
\int_M|G(p, t; y, s)-G(q, t; y, s)|dV_s(y)\leq 2\leq \frac{2r}{\sqrt{t}}. \eee
If $\frac{r}{\sqrt{t}}\leq \frac{1}{2\sqrt{B}}$, then let $\delta=\frac{r\sqrt{\a B}}{\sqrt{t}}$ and $\a=2$. So $\delta^2=\frac{r^2\a B}{t}\leq \frac{1}{4}\a<1$, this means that $\delta<1$. So

\bee
\int_M|G(p, t; y, s)-G(q, t; y, s)|dV_s(y)\leq C_1\lf(\exp(C_2\delta)-1\ri)\leq C_3\delta\le C_4\frac{ r}{\sqrt{t}}.
\eee
Here $C_1\sim C_4$ are positive constants depending only on $m, H_0, H_1, k_0$ and $T$.

Therefore, we complete the proof of this corollary.
\end{proof}

\subsection{A generalized maximum principle}
In this subsection, we want to show the following generalized maximum principle which will be used later frequently. This type of maximum principle  was originated by Karp and Li \cite{KarpLi}. Different variants were obtained   \cite{Liao-Tam,Ni-Tam-2004,Ecker-Huisken-1991}. We obtain the following  generalization of a result in   \cite{Ni-Tam-2004} using a trick in \cite{Ecker-Huisken-1991}.

\begin{thm}\label{max} Let $ g(x, t), t\in[0, T_1]$ be a family of smooth Riemannian metrics on $M^m$, with $\ppt g=H$, so that { $\sup\limits_{M\times[0, T_1]}|H|\leq R_0$}. Suppose $f(x, t)$ is a smooth function such that $\heat f\leq 0$ whenever $f\geq 0$ and \be\label{integration}
\int^{T_1}_0\int_{M}\exp{(-ar_0^2(x))}f_+^2(x, t)dV_0dt<\infty \ee for some constant $a>0$, where $r_0(x)$ is the distance function to a fixed point $p$ with respect to $g(0)$. If $f(x, 0)\leq 0$ for all $x\in M$, then  $f(x, t)\leq 0$ for all $(x, t)\in M\times[0, T_1]$.
\end{thm}

\begin{proof} In \cite{Ni-Tam-2004}, it was assumed that $\ppt g\le 0$. To prove the result in our setting, let $F(x, t)$ be such that $dV_t=e^{F(x, t)}dV_0$.
For $0<T\leq T_1$ which will be specified later and let \bee
 h(x, t)=-\frac{\theta r^2_t(x)}{4(2T-t)} \eee for $0\leq t\leq T$. Here $\theta>0$ is a constant which will be chosen later and $r_t(x)$ is the distance function to a fixed point $p$ with respect to $g(t)$. Then
 \bee\begin{split}
 \ppt h=&-\frac{\theta r^2_t(x)}{4(2T-t)^2}-\frac{\theta r_t(x)}{2(2T-t)}\cdot(\ppt r_t)\\
=&-\theta^{-1}|\n h|^2-\frac{\theta r_t(x)}{2(2T-t)}\cdot(\ppt r_t)\\
\le&-\theta^{-1}|\n h|^2+\theta^{-1}(2T-t)R_0|\n h|^2\end{split}\eee
because \bee
 |\n h|^2=\frac{\theta^2 r^2_t(x)}{4(2T-t)^2}\eee
 and   $|\ppt r_t|\leq \frac{1}{2}R_0 r_t$.
 Then \bee
\ppt dV_t=(\ppt F)dV_t=\frac{\tr_gH}{2}dV_t. \eee
Now we assume $T\leq \frac{1}{4R_0}$ and choose $\theta=\frac 14$, we obtain \be\label{evolution-g}\ppt h\leq -2|\n h|^2\ee for $0\leq t\leq T$.

Next, let $\beta>0$ be a constant and $0\le \phi(x)\le 1$ be
be the smooth function such that $\phi=1$ in $B_0(p, R)$, $\phi=0$ outside $B_0(p, 2R)$ and $|\wt\n\phi|\leq \frac 2R$, where  $\wt\n$ denotes the gradient with respect to $g(0)$.  We have \be\label{integration-heat}\begin{split}
 0\geq&\int^T_0e^{-\beta t}\int_M \phi^2e^hf_+\heat fdV_tdt \\
 =&\frac{1}{2}\int^T_0e^{-\beta t}\int_M \phi^2e^h\ppt(f_+^2)dV_tdt-\int^T_0e^{-\beta t}\int_M \phi^2e^hf_+(\Delta_t f)dV_tdt. \end{split}\ee
Here and in the following, $f_+=\max\{0, f\}$. Now we compute \eqref{integration-heat} term by term.

 By integration by part and Cauchy-Schwartz inequality, we have \be\label{T1}\begin{split}
 \int_M \phi^2e^hf_+(\Delta_t f)dV_t\leq&  \int_M e^hf_+^2|\n\phi|^2dV_t+ \int_M\phi^2e^hf_+^2|\n h|^2dV_t\\
 \le & \int_M e^hf_+^2|\n\phi|^2dV_t-\int_M\phi^2e^hf_+^2\ppt hdV_t.\end{split}\ee

 On the other hand, we have \be\label{T2}\begin{split}
&\frac{1}{2}\int^T_0e^{-\beta t}\int_M \phi^2e^h\ppt(f_+^2)dV_tdt\\
=&\frac 12 \bigg[(e^{-\beta t}\int_M\phi^2e^hf_+^2dV_t)|^T_0-\int^T_0e^{-\beta t}\int_M \phi^2e^h(\ppt h)f_+^2dV_tdt\\
&-\int^T_0e^{-\beta t}\int_M \phi^2e^hf_+^2(\ppt F)dV_tdt+\beta\int^T_0e^{-\beta t}\int_M \phi^2e^hf_+^2dV_tdt\bigg].\end{split}\ee
Since $|\ppt F|\leq C_1(n, R_0)$ for some constant depending only on $n, R_0$, if we choose $\beta=C_1(n, R_0)$, then by
by \eqref{evolution-g}, \eqref{integration-heat},  \eqref{T1}, \eqref{T2},
  we have\be\label{R}\begin{split}
  \int_M\phi^2(x)e^{h(x, T)}f_+^2(x, T)dV_T
\leq &4e^{\beta T}\int^T_0e^{-\beta t}\int_M e^hf_+^2|\n\phi|^2dV_tdt\\
\leq &C(n, R_0, T_1)e^{\beta T}\int^T_0\int_M e^hf_+^2|\wt\n\phi|^2dV_0dt.\end{split}\ee
Let $R\to\infty$ in \eqref{R}, we have \bee
 \int_Me^{h(x, T)}f_+^2(x, T)dV_T\leq \liminf\limits_{R\to\infty}\frac{C(n, R_0, T_1)e^{\beta T}}{R^2}\int^T_0\int_{B_0(p, 2R)-B_0(p, R)}e^{-\frac{r_0^2(x)}{C(R_0, T_1)T}}f_+^2dV_0dt. \eee
Hence if $T<\frac{1}{aC(R_0, T_1)}$, by the assumption \eqref{integration}, we have\bee
 \int_Me^{h(x, T)}f_+^2(x, T)dV_T\leq 0.\eee This implies $f(x, T)\leq 0$ for all $x\in M$. We can repeat the argument above to show that $f\le 0$ in $[0,T)$ if $T<\frac{2}{aC(R_0, T_1)}$. One then can start with $T$ and show that $f\le 0$ in $[0,2T)$ as long as $2T<T_1$. From this, it is easy to see that the theorem is true.

\end{proof}

\section{Results on Heat equation}\label{s-basicheat}
\subsection{Linear equation}\label{s-linear}

To prepare the construction of harmonic heat flow we first study the linear heat equation.
Let $M^m$ be a non-compact smooth manifold with dimension $m\ge 3$ and let   $g(t)$ be a family of smooth complete Riemannian metrics on $M$, $0\le t\le T$ for some $T>0$. More precisely, $g(t)$ is  smooth both in space and time on $M\times[0,T]$.
Denote
\be\label{e-pg}
H(x,t):=\frac{\p}{\p t}g(x,t).
\ee
Let $F(x,t)$ be a bounded smooth function on $M\times[0,T]$ and $f(x)$ be a bounded smooth function on $M$. We want to study the following problems:
\be\label{e-NH}
 \left\{
  \begin{array}{ll}
   \dps {\heat u}=F & \hbox{in $M\times[0,T]$;}  \\
    u(x,0)=0;
  \end{array}
\right.
\ee
and
\be\label{e-H}
 \left\{
  \begin{array}{ll}
    \dps{\heat v}=0 & \hbox{in $M\times[0,T]$;} \\
    v(x,0)=f(x).
  \end{array}
\right.
\ee
Here $\Delta_{g(t)}$ is the Laplacian operator with respect to $g(t)$.

 \begin{prop}\label{p-heat} With the above notation and assumptions,
   there is a solution $u$ of \eqref{e-NH} and a solution $v$ of \eqref{e-H} so that both $u$ and $v$ are smooth in $M\times[0,T]$. Moreover,
       \bee
       \left\{
         \begin{array}{ll}
           \sup\limits_{M\times[0,T]}|u|\le T\sup\limits_{M\times[0,T]}|F|;\\
          \sup\limits_{M\times [0,T]}|v|\le \sup\limits_M|f|.
         \end{array}
       \right.
       \eee

 \end{prop}
 \begin{proof} This  is standard. For any $R>>1$, let $0\le \phi_R\le 1$ be a smooth function on $M$ so that $\phi_R=1$ in $B_p(R)$ and $\phi_R=0$ outside $B_p(2R)$ where $p\in M$ is a fixed point and $B_p(r)$ is the geodesic ball of radius $r$ with respect to $g(0)$. By \cite[Theorems 7, 12, Chapter 3]{Friedman}, there is a smooth solution $u_R$ of the following initial-boundary value problem
 \bee
 \left\{
  \begin{array}{ll}
    \heat u_R=\phi_RF & \hbox{in $\Omega_R\times[0,T]$;} \\
    u_R(x,0)=0& \hbox{in $x\in\Omega_R$;}\\
    u_R(x,t)=0& \hbox{in $(x,t)\in\p\Omega_R\times[0,T]$.}
  \end{array}
\right.
 \eee
 where $\Omega_R$ is a bounded domain in $M$ with smooth boundary and  with $B_p(2R)\Subset \Omega_R$.
Since
 \bee
\lf|\heat u_R\ri|\le \sup_{M\times[0,T]}|F|.
 \eee
Let $\mathfrak{m}=\sup_{M\times[0,T]}|F|$. Then  $$\heat(u_R-t\mathfrak{m})\le 0, \ \ \heat (u_R+t\mathfrak{m})\ge0.
$$
By the maximum principle, one can conclude that:
 $$
 \sup_{\Omega_R\times[0,T]}|u_R|\le T\sup_{M\times[0,T]}|F|.
 $$
 From this one may conclude that for any bounded domain $D$ in $M$, and for any $k\ge1$, the derivatives of $u_R$ with respect to space up to order $k$ and the derivatives with respect to $t$ up to order $[k/2]$ are bounded  in $D\times[0,T]$ by a constant independent of $R$, provided $R$ is large enough. Here $[k/2]$ is the integral part of $k/2$. See \cite[Chapter 4]{LSU} for example. From this, by taking a convergent subsequence, one can find a smooth solution of \eqref{e-NH} so that
 \bee
 \sup_{M\times[0,T]}|u(x,t)|\le T\sup_{M\times[0,T]}|F|.
 \eee

The construction of solution $v$ to \eqref{e-H} with the following estimate is similar:
 $$\sup_{M\times [0,T]}|v|\le \sup_M|f|.
 $$

 \end{proof}

To construct harmonic map heat flow,
we also need some estimates of the gradients of the solutions obtained in the previous proposition. In order to obtain the estimates, we need more conditions on $g(t)$. As before, let
$$
H=\frac{\p}{\p t}g.
$$
\begin{prop}\label{p-heat-grad} With the   notation and assumptions as in Proposition \ref{p-heat}. Moreover, assume that
\bee
|H|_{g(t)}, |\nabla H|_{g(t)}, |\Rm(g(t))|_{g(t)}\le K
\eee
for some $K>0$ on $M\times[0,T]$.
 \begin{itemize}
 \item[(i)] The solutions $u, v$ obtained in Proposition \ref{p-heat} satisfy the following gradient estimates:
$$
\sup_{M}|\nabla u|(\cdot, t)\le C(m, K, T)\lf(\sup_{M\times[0,t]}|F|\ri)t^\frac12
$$
and
$$
\sup_M|\nabla v|(\cdot,t)\le e^{C(m, K)t}\sup_M |\nabla f|
$$ for all $0\le t\le T$, for some constants $C(m,K)$ depending only on $m, K$ and $C(m,K,T)$ depending only on $m, K, T$.

\item[(ii)] The solution $v$ obtained in Proposition \ref{p-heat} satisfies the following estimate:
$$
|v(x,t)-f(x)|\le C(m, K, T) t^\frac12\sup_M|\n f|
$$
for all $(x,t)\in M\times[0,T]$ for some constant $C(m,K, T)$ depending only on $m, K$ and $T$.
\end{itemize}

 \end{prop}
\begin{proof} (i) Let us prove the estimate of $|\nabla v|$ first. By the Bochner formula and the fact that $|H|$ and $|\Rm(g(t))|$ are uniformly bounded by $K$, one can conclude that
$$
\heat |\nabla v|\le C(m, K)|\nabla v|.
$$
whenever $|\nabla v|>0$.  So we have
$$
\heat e^{-C(m, K)t}|\nabla v|\le 0.
$$
On the other hand, since $v$ is bounded and
\bee
\heat v=0,
\eee
one can conclude by using cutoff functions and integrating by parts  that
\bee
\int_0^T\int_M \exp (-ar_0^2(x))|\nabla v|^2 dV_0dt<\infty.
\eee
for some $a>0$. Here we have used the fact that $|H|$ is bounded so that $g(t)$ and $g_0$ are uniformly equivalent and volume comparison because $|\Rm(g(0))|$ is bounded. Apply the maximum principle Theorem \ref{max} to the function
$$
 e^{-C(m, K)t}|\nabla v|-\sup_M |\nabla f|,
 $$
 one can conclude that
 $$
 |\nabla v|(x,t)\le e^{C(m, K)t}\sup_M |\nabla f|.
 $$
 in $M\times[0,T]$.

  Next we want to estimate   $|\nabla u|$. Since $|H|, |\nabla H|, |\Rm(g(t))|$ are bounded by $K$, we can construct the fundamental solution
 $G(x,t;y,s)$ of the heat operator as in Section \ref{S-Green} with estimates as in \cite{Chau-Tam-Yu}. If we let
 $$
 w(x,t)=  \int^t_0\int_M G(x, t; y, s)F(y, s)dV_s(y)ds,
 $$
 then $\heat w=F$ in $M\times(0,T]$ which is continuous up to $t=0$ so that $w(x,0)=0$. Moreover, $w$ is bounded by Theorem \ref{int-G}. By the maximum principle Theorem \ref{max}, we conclude that $u\equiv w$. Hence
 $$
 u(x,t)=\int^t_0\int_M G(x, t; y, s)F(y, s)dV_s(y)ds.
 $$

 Then by Corollary \ref{c-G} we have:

 \bee\begin{split}
 |u(x, t)-u(x', t)|\leq& \int^t_0ds\int_M|G(x, t; y, s)-G(x', t; y, s)|\cdot|F(y, s)|dV_s(y)\\
 \leq& (\sup\limits_{M\times[0, t]}|F|)\cdot \int^t_0ds\int_M|G(x, t; y, s)-G(x', t; y, s)|dV_s(y)\\
 \leq& C(\sup\limits_{M\times[0, t]}|F|)\cdot \int^t_0\frac{r_t(x, x')}{\sqrt{t-s}}ds\\
 \leq& Cr_t(x, x')\cdot(\sup\limits_{M\times[0, t]}|F|)\cdot t^{\frac 12}. \end{split}
 \eee
 From this, it is easy to see that the estimate for $|\nabla u|$ is true.

 To prove (ii),   for $x\in M$,

\bee\begin{split}
&|v(x,t)-f(x)|\\
=
&\lf|\int_{M}G(x, t; y, 0)f(y)dV_{0}(y)-f(x)\ri|\\
=&\lf|\int_{M}G(x, t; y, 0)f(y)-f(x)dV_{0}(y)\ri|\\
\le &\sup_M|\n f|\int_{M}G(x, t; y, 0)r(x, y)dV_0(y)\\
 \end{split}\eee
where $r(x,y)$ is the distance between $x, y$ with respect to $g(0)$ and we have used Theorem \ref{int-G}. By Theorem \ref{int-G} and volume comparison, one can proceed as in \cite{Li-Tam} to conclude that
\bee
\int_{M}G(x, t; y, 0)r(x, y)dV_0(y)\le C_1t^\frac12
\eee
for some constant $C_1$ depending only on $m, K, T$.  From this, (ii) follows.

\end{proof}

\subsection{A semi-linear heat equation}\label{s-semilinear}
We want to use the results in \S \ref{s-linear} to study the following semi-linear equation.
Let $g(t)$ be a smooth family of complete metrics defined on $M$ with $t\in[0,T]$. We want to consider the following system of semi-linear equation which is closely related to harmonic map heat flow:
\be\label{e-semi-linear}
\left\{
  \begin{array}{ll}
    \heat u=F_{BC}(u)\la \nabla  u^B,\nabla  u^C\ra    & \hbox{in $M\times[0,T]$ ;} \\
    u(0,x)=f(x),
  \end{array}
\right.
\ee
where $u=(u^A):M\times[0,T]\to \R^q$ is a vector-valued function and $f=(f^A):M\to\R^q$ and $F_{BC}=(F^A_{BC}):  \R^q\to \R^q$ are smooth functions. The $\nabla u^B$ and the inner product $\la \nabla  u^B,\nabla  u^C\ra$  are taken with respect to $g(t)$. As before, let
$$
H:=\frac{\p}{\p t}g.
$$

\begin{lma}\label{l-semi} Assume
\bee
|H|_{g(t)}, |\nabla H|_{g(t)}, |\Rm(g(t))|_{g(t)}\le K
\eee
for some $K>0$ on $M\times[0,T]$ and  $|F|\le L$ .
Suppose $f$ is a smooth function so that $f$ and $|\nabla f|$ are bounded with \bee\mathfrak{m}=\sup_M\lf(\sum_A|\nabla f^{A}|^2 \ri)^\frac12<\infty.\eee Then there is a constant  $T_1>0$ depending only on $m, q, K, L, T$ and $\mathfrak{m}$  so that  \eqref{e-semi-linear} has a smooth solution in $M\times[0,T_1]$  with $u$ and $|\nabla u|$ uniformly bounded.

\end{lma}
\begin{proof} We use iteration as in \cite{E-S,Li-Tam}.  Define $u^{-1}=0$ and define $u^k$ inductively: $u^k$ is the solution of the following linear equation:
\be\label{e-semi-iterate}
\left\{
  \begin{array}{ll}
    \heat u^k=F_{BC}(u^{k-1})\la \nabla  u^{k-1,B},\nabla  u^{k-1,C}\ra     & \hbox{in $M\times[0,T]$ ;} \\
    u^k(0,x)=f(x),
  \end{array}
\right.
\ee
for $k\ge 0$ where $u^k=(u^{k,A})$. The equation for each component is:
$$
\heat u^{k,A}=F_{BC}^A(u^{k-1})\la \nabla  u^{k-1,B},\nabla  u^{k-1,C}\ra .
$$

First we want to show that $u^k$ is well-defined and smooth in $M\times[0,T]$  for all $k\ge0$.  Suppose $u^{k-1}$ is well-defined and smooth  so that
\bee
\sup_{M\times [0,T]}|\nabla u^{k-1}|^2<\infty.
\eee
Note that this is true for $k=1$, by Proposition \ref{p-heat}, Proposition \ref{p-heat-grad} and  assumptions on $f$.

Since $|F|$ is bounded and the inductive hypothesis,   by Proposition \ref{p-heat} and Proposition \ref{p-heat-grad}, then \eqref{e-semi-iterate} has a solution $u^k$ which is smooth in $M\times[0,T]$, is uniformly bounded and \bee
\sup_{M\times [0,T]}|\nabla u^{k}|^2<\infty.
\eee

Next we want to show that if $0<T_1\leq T$ is small enough, then $|\nabla u^k|$ will be uniformly bounded independent of $k$ in $M\times[0,T_1]$.   By Proposition \ref{p-heat-grad}, we have\bee
|\nabla u^{k,A}|(\cdot, t)\le C(m, K, T) t^\frac12\sup_{M\times[0,t]}|F^A_{BC}(u^{k-1})||\nabla  u^{k-1,B}|\,|\nabla  u^{k-1,C}|+e^{C(m ,K)t}\sup_M|\nabla f^A|
\eee

Let
$$
p_k(t)=\sup_{M\times[0, t]}\lf(\sum_A|\nabla u^{k,A}|^2(\cdot,t)\ri)^\frac12
$$
and let
$$
\mathfrak{m}=\sup_M\lf(\sum_A|\nabla f^{A}|^2 \ri)^\frac12.
$$
Then we have
\bee
p_k(t)\le C_1(t^\frac12 p_{k-1}^2+\mathfrak{m})
\eee
for some constant $C_1$ depending only on $m, q, K, T, L$. So
\bee
C_1t^\frac12p_k(t)\le\lf( C_1t^\frac12 p_{k-1}(t)\ri)^2+C_1^2t^\frac12 \mathfrak{m}.
\eee
Suppose $T_1$ is such that $C_1^2T_1^\frac12 \mathfrak{m}\le \frac14$, then for $0<t\leq T_1$
\bee
C_1t^\frac12p_0(t)\le \frac12.
\eee
Inductively, we conclude that
$$
C_1t^\frac12 p_k(t)\le \frac12.
$$
Hence we let $T_1>0$ so that $T_1^\frac12=\min\{T^\frac12, \frac14 C_1^{-2}\lf(1+\mathfrak{m}\ri)^{-1}\}$, then
$$
p_k(T_1)\le \frac12C_1^{-1}T_1^{-\frac12}
$$
for all $k$. From this and by the proof of Proposition \ref{p-heat}, we also conclude that $u^k$ are uniformly bounded on $M\times[0,T_1]$.

 We claim that in any bounded coordinate neighborhood $U$,   for any $l\ge 1$, there is a constant $C$ independent of $k$ so that $|D^\a_tD_x^\b u^k|\le C$ if $2\a+\b\le l$. Here $D_t$ and $D_x$ are partial derivatives with respect to $t$ and local coordinates $x$. If the claim is true, then by a diagonal process, we can find  a smooth solution of \eqref{e-semi-linear} in $M\times[0,T_1]$,  so that $|u|$ and $|\nabla u|$ are uniform bounded in space and time.

The idea of the claim is as follows. For each $k$, the RHS of \eqref{e-semi-iterate} are uniformly bounded. By standard theory, we have some H\"older norm of the $\nabla u^k$ being bounded. This will imply bounds of higher derivatives for $u^{k+1}$ etc. We sketch the proof as follows. Let $\phi$ be a smooth cutoff function with support inside a bounded coordinate neighborhood $U$ so that it is 1 in an open set $V\Subset U$. Then one can check that
\bee
\heat (\phi (u^k-f))=G^k
\eee
where $G^k$ is uniformly bounded by a constant independent of $k$ and is zero outside $U$. Moreover, $\phi(u^k-f)=0$ at $t=0$. By \cite[Theorem 4, p.191]{Friedman}, we have
$$
|u^k|_\delta+|D_xu^k|_\delta\le C
$$
in $V\times[0,T_1]$ for some constant $C$ and $\delta>0$ independent of $k$. Here $|\cdot|_\delta$ is the H\"older norm in $V\times[0,T_1]$ with respect to the distance function $d(P,Q)=\lf(|x-x'|^2+|t-t'|\ri)^\frac12$ for $P=(x,t), Q=(x',t')$. From this and the Schauder estimates, one may get $|u^{k+1}|_{2+\delta}$ being uniformly bounded in $V'\times [0,T_1]$ for any $V'\Subset V$. Then  $|u^{k+2}|_{4+\delta}$ is uniformly bounded in $V''\Subset V'$ and so on. This proves the claim.

\end{proof}

\section{Short time existence of Harmonic map heat flow}\label{s-harmonic-heat}

We will obtain  a short time existence result for harmonic map heat flow coupled with a smooth family of complete metrics in the domain manifold. First, let us recall the basic facts about the harmonic map heat flow.

\subsection{The harmonic map heat flow}\label{s-setup}

Let $(M^m, g)$ and $(N^n, h)$ be two Riemannian manifolds and  $f: (M^m, g)\to (N^n, h)$ be a smooth map. Let $\nabla, \wt\nabla$ be Riemannian connections on $M, N$ respectively. Consider the vector bundle $T^*(M)\otimes f^{-1}(T(N))$. Let $D$ be the connection on this bundle defined as (for $\omega$ a 1-form and $Y$ a vector field along $f$):
$$
D_X(\omega\otimes Y)=\nabla_X\omega\otimes Y+\omega\otimes \wt\nabla_{f_*X}Y.
$$
In general, one can extend the connection  to $\otimes^k(T^*(M))\otimes f^{-1}(T(N))$. If in local coordinates $x$ in $M$, $y$ in $N$, a section of this bundle is given by
$$
s=u^\a_{i_1\dots i_k}dx^{i_1}\otimes\dots\otimes dx^{i_k}\otimes \p_{y^\a},
$$
then
\bee
\begin{split}
s_{|p}=&
D_{\p_{ x^p}}s\\
=&u^\a_{i_1\dots i_k;p} dx^{i_1}\otimes\dots\otimes dx^{i_k}\otimes \p_{y^\a}+u^\a_{i_1\dots i_k}dx^{i_1}\otimes\dots\otimes dx^{i_k}\otimes\wt\nabla_{f_*(\p _{x^p})} \p_{y^\a}\\
=&u^\a_{i_1\dots i_k;p} dx^{i_1}\otimes\dots\otimes dx^{i_k}\otimes \p_{y^\a}+f^\b_pu^\a_{i_1\dots i_k}dx^{i_1}\otimes\dots\otimes dx^{i_k}\otimes\wt\nabla_{\p_{y^\b}} \p_{y^\a}.
\end{split}
\eee

Here and in the following  $;$ denotes the covariant derivative with respect to $\n$ and $|$ denotes the covariant derivative with respect to the connection $D$ on the bundle $T^*(M)\otimes f^{-1}(T(N))$.

In case we have a smooth map $f: M\times[0,T]\to N$, we may also consider $D_t=D_{\p_t}$. If
$$
s=u^\a_{i_1\dots i_k}dx^{i_1}\otimes\dots\otimes dx^{i_k}\otimes \p_{y^\a}
$$
then
$$
s_{|t}=D_t(s)=\p_t u^\a_{i_1\dots i_k}dx^{i_1}\otimes\dots\otimes dx^{i_k}\otimes \p_{y^\a}+f^\b _t u^\a_{i_1\dots i_k}dx^{i_1}\otimes\dots\otimes dx^{i_k}\otimes\wt\nabla_{\p_{y^\b}}  \p_{y^\a}.
$$

Now consider a  smooth map $f:M\times [0,T]\to N$ and its derivative
$$
s=:df=f^\a_i dx^i\otimes \p_{y^\a}.
$$

The energy density of $f$ is defined by
$$e(f):=|s|^2_{g,h}:=g^{ij}f^\a_if^\b_j h_{\a\b}
 $$
 in local coordinates.
 The second fundamental form of $f$ is defined by $$Ds:=Ddf.$$ In local coordinates,
 \bee
\begin{split}
Ds=D df = & s^\a_{i|j} dx^i\otimes dx^j\otimes \p_{y^\a}\\
=&f^\a_{;ij} dx^i\otimes dx^j\otimes\p_{y^\a}+ f^\a_i f_j^\b dx^i\otimes dx^j\otimes \wt\nabla_{\p_{y^\b}}\p_{y^\a}.\\
\end{split}
\eee
Note that $s^\a_{i|j}=s^\a_{j|i}$.
The tension field of $f$ is defined by $$\tau(f):=\tr_g(Ds),$$
 which is the trace of  the second fundamental form.  In local coordinates

$$
\tau(f)^\a:=g^{ij}s^\a_{i|j}.
 $$
 Suppose $g(t)$ is a smooth family of metrics on $M$, $t\in [0,T]$. Then the harmonic map heat flow $f(x, t)$ coupled with varying metrics $g(t)$ is defined by
 \be\label{harmonic-map-flow}\ppt f=\tau(f).\ee
 Here $f:M\times[0,T]\to N$ is a smooth map and the tension field on the right is computed with $g(t)$. See the seminal paper  by Eells and Sampson \cite{E-S}. Note that in \cite{E-S}, the metric $g$ is fixed.

In local coordinates, \be\label{e-harmoinic-map-flow-2}
\ppt f^\a(x, t)=g^{ij}(x, t)(f^\a_{ij}-\Gamma^k_{ij}f^\a_k+\wt\Gamma^\a_{\b\gamma}f^\b_if^\gamma_j).\ee
Here $f^\a_i, f^\a_{ij}$ denote  the partial derivatives of $f^\a$ and $\Gamma, \wt \Gamma$ are the connections of $g(t)$ and $h$ respectively.

\subsection{A priori estimates}\label{s-estimates}
We want to obtain some a priori estimates for the energy density and the norm of the tension field for solutions of harmonic map heat flow. Let us first estimate the energy density. Let $g(t)$ be a smooth family of complete metrics on $M^m$ which is noncompact, $t\in [0,T]$ and let $(N^n,h)$ be another complete Riemannian manifold. Suppose
$$
F:M\times[0,T]\to N
$$
is a solution to the harmonic map heat flow.
As before, let
$$
H=\frac{\p}{\p t}g.
$$
In the following,   $R_{1221}$ is the sectional curvature for an orthonormal pair of vectors.
Direct computations give:

\begin{lma}\label{l-energy} In local coordinates of  $x^i$ in $M$ and $y^\a$ in $N$,
\bee
\begin{split}
\heat e(F)=&- g^{il}g^{kj}\lf(H_{kl}+2R_{kl}\ri)F^\a_iF^\b_j h_{\a\b}-2|DdF|^2 \\
&+2g^{pq}g^{ij}F^\sigma_p F^\gamma_qF^\tau_i F^\b_j S_{\gamma\tau\b\sigma}.\end{split}
\eee
where $R_{ij}$ is Ricci tensor of $g(t)$ and    $S$ is the curvature tensor of $N$.
\end{lma}
\begin{proof} This is well-known \cite{E-S}. The only difference is that $g$ also depends on $t$, and we have a  term involving $\p_t g=H$.
\end{proof}

\begin{lma}\label{l-e-energy}
Let $(M^m, g(t))$, $(N,h)$ and $F$ be as in the previous lemma so that $e(F)$ is uniformly bounded in space and time. Suppose
  that $|H|_{g(t)} $ is uniformly bounded by $L$. Suppose $$2\Ric(g(t))(x,t)+H(x,t)\ge -K(t)g(x, t)$$
for some $K(t)\ge0$ so that $$K_0=:\int_0^TK(t)dt<\infty,  $$
and suppose $\Rm(h)\le \kappa$ for some $\kappa\ge 0$.
Let
\bee
e_0=\sup_M e(F)(\cdot, 0).
\eee

 Then
\bee
e(F)(\cdot,t)\le \exp(\lambda(t))v(t)
\eee
on $[0, T_1]$ where
\bee
v(t)=\lf(e_0^{-1}-2\kappa \exp(K_0)t\ri)^{-1}, \quad\lambda(t):=\int^t_0K(\tau)d\tau\eee
and $T_1=\min\{T,{\frac 12}\lf( 2\kappa e_0 \exp(K_0)\ri)^{-1}\}$.   Hence, $e(F)(\cdot,t)\le 2e_0\exp(K_0)$ for $t\in [0,T_1]$.  In particular if $\kappa=0$, then $T_1=T$.
\end{lma}
\begin{proof} Let $s=dF=F^\a_i dx^i\otimes\p_{y^\a}$ in local coordinates of $x\in M$ and $F(x,\cdot)$ in $N$. Let $S$ be the curvature tensor of $(N,h)$, then by Lemma \ref{l-energy}, we have
\bee (\ppt-\Delta_{  g(t)}) e(F)\le K(t) e(F)+2\kappa e^2(F).
\eee

 Let  $\lambda(t):=\int^t_0K(\tau)d\tau\leq K_0<+\infty.$ Then

\bee
\begin{split}
(\ppt-\Delta_{  g(t)})(\exp( -\lambda)e(F))\le & 2\kappa \exp(\lambda)(\exp( -\lambda)e(F))^2\\
\le& 2\kappa \exp(K_0)(\exp( -\lambda)e(F))^2
\end{split}
 \eee
Let $v(t)$ be the solution of the ODE
$$
v'=2\kappa \exp(K_0)v^2
$$
with $v(0)=e_0$. Then
$$
v(t)=\lf(e_0^{-1}-2\kappa \exp(K_0)t\ri)^{-1}.
$$
$v(t)$ is well-defined if $t< \lf( 2\kappa e_0 \exp(K_0)\ri)^{-1}$. Let $T_1=\min\{T, \frac 12\lf( 2\kappa e_0 \exp(K_0)\ri)^{-1}\}$.

Let   $\Theta(x, t):=\exp({-\lambda(t)})e(F)$. Then in $M\times[0,T_1]$,  we have
\bee
(\ppt-\Delta_{  g(t)})(\Theta -v )\leq 2\kappa \exp(K_0)(\Theta+v)\cdot(\Theta-v)\le C_1(\Theta-v)
\eee
for some constant  $C_1>0$ whenever $\Theta-v>0$.  Since $g(t)$ are uniformly equivalent to $g(0)$ and the Ricci curvature of $g(t)$ is bounded from below, one may apply the maximum principle Theorem \ref{max} to conclude that $\Theta\le v$ in $M\times[0, T_1]$. The result follows.

\end{proof}
We should remark that the  bounds of $e(F)$ and $T_1$ do not depend on $L$.

\vskip .1cm

In order to study the distance between $F(x,t)$ and the initial map $F(x,0)$, we need to estimate the norm of the tension field. Again by direct computations we have, see \cite{Hartman}:

\begin{lma}\label{l-tension}
\bee\label{evolution-t}
\heat |\tau(F)|^2=2S_{\delta\a\gamma\b}F^\delta_kF^\gamma_kF^\a_tF^\b_t-2F^\a_{tk}F^\a_{tk}-2F^\a_tH_{kl}F^\a_{kl}-F^\a_tF^\a_k(2\n_lH_{lk}+\n_kH_{ll}).
\eee
Here the computation is at $x$  and $F(x, t)$ under normal coordinates with respect to $g(t)$ and $h$.
\end{lma}
Using this we   obtain the following:

\begin{lma}\label{l-e-tension} With same notation and assumptions as in Lemma \ref{l-e-energy}. In additions, assume $|H|\le at^{-1}, |\nabla H|\le at^{-\frac32}$ for some $a>0$. Suppose
$$
e(F)\le \mathfrak{m}.
$$
in $M\times[0,T]$.
 Then there is a constant $C>0$ depending only on $m, n, T, a, K_0, \kappa, \mathfrak{m}$ such that
$$
|\tau(F)|(x,t)\le Ct^{-\frac12}
$$
on $M\times[0, T]$.
\end{lma}

\begin{proof}

By Lemma \ref{l-tension}, we have
 \bee
 \begin{split}
\heat |\tau(F)|^2\leq & 2S_{\delta\a\gamma\b}F^\delta_kF^\gamma_kF^\a_tF^\b_t-2F^\a_{tk}F^\a_{tk}-2F^\a_tH_{kl}F^\a_{kl}-F^\a_tF^\a_k(2\n_lH_{lk}+\n_kH_{ll})\\
\le& C(m, n)\lf(\kappa e(F)|\tau(F)|^2+ at^{-1}|\tau( F)||Dd F|+at^{-\frac32}e(F)^{\frac 12}|\tau(F)|\ri)\\
& -2F^\a_{tk}F^\a_{tk}.
\end{split}
\eee
 So at the point where $|\tau(F)|>0$,
\bee
\heat |\tau(F)|\le C(m, n)\lf(\kappa e(F)|\tau(F)| + at^{-1} |Dd F|+at^{-\frac32}e(F)^{\frac 12} \ri)
\eee
and
\bee
\begin{split}
\heat(t|\tau(F)|)\le & C(m, n) \lf(\kappa e(F)t|\tau(F)| + a  |Dd F|+at^{-\frac12}e(F)^{\frac 12} \ri)+|\tau (F)|\\
\le &C_1( |DdF|+t^{-\frac12}).
\end{split}
\eee
because $|\tau(F)|\le |DdF|$ and $t\le T$. Here and below $C_i$ will denote a positive constant depending only on $m, n, T, K_0, \kappa, a, \mathfrak{m}$.

On the other hand, by Lemma \ref{l-energy}, we have
\bee
\heat e(F)\le K(t)e(F)+2\kappa e^2(F)-2|DdF|^2.
\eee
Let $\lambda(t)=\int_0^tK(s)ds$ and $\Theta =\exp(-\lambda (t))e(F)$, then
\bee
(\ppt- \Delta)\Theta\le {-{C_2}^{-1}|DdF|^2+C_2}.
\eee
So
\bee
(\ppt- \Delta)t^{\frac12}\Theta\le -C_2^{-1}t^\frac12| DdF|^2+C_3t^{-\frac12}.
\eee
This implies
\bee
(\ppt- \Delta)\lf(t|\tau(F)|+t^{\frac12}\Theta\ri)\le C_4t^{-\frac12}
\eee
Hence
\bee
(\ppt- \Delta)\lf(t|\tau(F)|+t^{\frac12}\Theta-C_5t^\frac12\ri)\le 0.
\eee

Since we do not assume $|\tau(F)|$ is bounded, we need to estimate the integral of $|\tau(F)|^2$ in order to apply   the maximum principle.
Recall
\bee
(\ppt- \Delta)\Theta\le  { -{C_2}^{-1}|  DdF|^2+C_2}
\eee
Multiplying a cutoff function to the above inequality and then integrating by part, {one can prove that
$$
\int_0^{T_1}\int_{B_0(R)}|DdF|^2 dV_0dt\le CV_{g(0)}(2R)
$$}
 for some constant $C$ independent of $R$. Here we have used the fact that  $\Theta$ is uniformly bounded and $g(t)$ are uniformly equivalent to $g(0)$. Using the fact that $V_{g(0)}(2R)\le \exp(C'(R+1))$, for some $C'>0$ independent of $R$, the lemma follows from the maximum principle Theorem \ref{max}.
\end{proof}

\begin{rem} In the above lemma, we do not assume that $|\tau(F)|$ is bounded. In particular, we do not assume the tension field of the initial data is bounded.
\end{rem}

 \subsection{Short-time Existence}\label{existence}

 Let $M^m$ be a noncompact manifold and let $g(t)$ be a smooth family of complete metrics defined on $M\times[0,T]$ so that
\be\label{e-g-flow}
\frac{\p }{\p t}g(x,t)=H(x,t).
\ee
Let $ (N^n,h)$  be another complete Riemannian manifold. Consider the following assumptions.

\vskip .3cm

\begin{itemize}
  \item[{\bf (a1)}] $2\Ric(g(t))+H(t)\ge -K(t)g(t)$ in $M\times [0,T]$ where $K(t)\ge0$ and
  $$K_0=:\int_0^TK(t)dt<\infty.
      $$
  \item[{\bf (a2)}] $|H|\le at^{-1}$  and $|\nabla H|\le at^{-\frac32}$ for some $a>0$. Here the norm and the covariant derivative are with respect to $g(t)$.
  \item[{\bf (a3)}] The curvature of $h$ is bounded from above: $\Rm(h)\le \kappa$ for some $\kappa\ge0$..
\end{itemize}

\vskip .3cm

We will prove the main short time existence result of harmonic map heat flow Theorem \ref{t-existence-main} in this subsection.

As a corollary, we remove a condition that the image of the initial map is bounded in  the  short time existence  result  \cite[Theorem 3.4]{Li-Tam}, provided there is a suitable exhaustion function and curvature of  the target manifold is bounded from above.
\begin{cor}\label{c-LT}
Let $(M^m, g)$ be a complete noncompact Riemannian manifold with $\Ric(g)\ge -Kg$ for some $K\geq 0$ and let $(N^n,h)$ be another complete noncompact manifold with $\Rm(h)\le \kappa$ for some $\kappa\ge 0$. Suppose there is a smooth function $\gamma$ on $M$ satisfying:
$$
d (p,x)+1\leq\gamma(x)\leq  d(p,x)+C_0$$ and $$
|\nabla^k \gamma| \le C_0, \quad k=1,2$$
for some $C_0>0$ where $d(p,x)$ is the distance function on $M$ and $p\in M$ is a fixed point. Then for any smooth map $f: M\to N$ with energy density uniformly bounded by $e_0$, there exists a solution to the harmonic map  heat flow $F$ from $M\times[0,T_0]\to N$ with initial value $F(x,0)=f(x)$, where $T_0=  C_1\kappa^{-1}$ for some $C_1$ depending only on $e_0, K$.  Moreover,
\bee
\sup_M e(F(\cdot,t))\le 2e_0\exp(Kt).
\eee
In particular, if $\kappa=0$, then the heat flow has long time solution.
\end{cor}

We can also obtain a short time solution of the harmonic map heat flow coupled with the Ricci flow Theorem \ref{t-Ricci} mentioned in the introduction.

Before we prove  Theorem \ref{t-existence-main}, we need the following extension lemma which will be used later.

\begin{lma}\label{l-extension}
Let $g(t)$ be a smooth family of complete metrics on $M$, $t\in [0,T]$, and $(N,h)$ is another smooth complete manifold. Let $0<T_1<T$. Suppose $F_1$ is a smooth solution to the harmonic map heat flow from $M\times[0,T_1]$ to $N$ and $F_2$ is a smooth solution to the harmonic map heat flow  $M\times[T_1,T]$ to $N$. Suppose $F_1=F_2$ at $t=T_1$. Let
\bee
F(x,t)=\left\{
         \begin{array}{ll}
           F_1(x,t), & \hbox{if $(x,t)\in M\times[0,T_1]$;}\\
F_2(x,t), & \hbox{if $(x,t)\in M\times[T_1,T]$.}
         \end{array}
       \right.
\eee
Then $F$ is a smooth solution to the harmonic map heat flow from $M\times[0,T]$ to $N$.
\end{lma}
\begin{proof} It is sufficient to show that $F$ is smooth near $(p,T_1)$ for all $p.$  Consider local coordinates $x^i$ near a point $p$ in $M$ and $y^\a$ near the point $F(p, T_1)$ in $N$. Near $(p,T_1)$,
\bee
\begin{split}
\ppt F_1^\a(x, t)=&g^{ij}(x,t)((F_1 ^\a)_{ij}(x,t)\\
&-\Gamma^k_{ij}(x,t)(F_1^\a)_k(x,t)+\wt\Gamma^\a_{\b\gamma}(F_1(x,t))
(F_1^\b)_i(x,t)(F_1^\gamma)_j(x,t)).
\end{split} \eee
Here $(F_1^\a)_i, (F_1^\a)_{ij}$ denote  the partial derivatives of $F_1^\a$ and $\Gamma, \wt \Gamma$ are the Levi-Civita connections of Riemannian manifolds $(M, g(t))$ and $(N, h)$ respectively. Similarly, we have the corresponding equations for $F_2$. Since $F_1, F_2$ are smooth up to $T_1$, we conclude that as $(x,t), (x',t')\to (p,T_1)$ with $t>T_1>t'$,  all the corresponding space derivatives of $F_2(x,t)$ and $F_1(x',t')$ will converge to the same limit. From the equations, we conclude that $\p_tF_2(x,t)$ and $\p_tF_1(x',t')$ will converge to the same limit. Differentiate the equation with respect to $x$, we can conclude that $\p_t\p^l_x F_1(x',t')$ and $\p_t\p^l_x F_2(x ,t)$ will converge to the same limit. Differential the equation with respect to $t$ we conclude that $\p_t\p_tF_2(x,t)$ and $\p_t\p_tF_1(x',t')$ will converge to the same limit. Continue in this way, one can see that the lemma is true.

\end{proof}

The proof of Theorem \ref{t-existence-main} follows from the following special case so that condition {\bf (a1)} is replaced by a condition on $\Ric(g(t))$ and {\bf (a2)} is replaced by the conditions that $|H|, |\n H|$ are uniformly bounded.
\begin{prop}\label{p-existence-2}
Let $M^m$ be a noncompact manifold and let $g(t)$ be a smooth family of complete metrics defined on $M\times[0,T]$ so that
$$
\frac{\p }{\p t}g(x,t)=H(x,t).
$$
Let $(N^n,h)$ be another complete manifold.
 Suppose  {\bf  (a3)} is  satisfied. Assume   $|H|_{g(t)}\le L$, $|\nabla H|_{g(t)}\le L$    in $M\times[0,T]$ for some constant $L>0$ and $\Ric(g(t))\ge -K(t) g(t)$ for some $K(t)\ge0$ so that
 $$
 K_0:=\int_0^TK(t)dt<\infty.
  $$
  Moreover, assume  there exists a smooth exhaustion function $\gamma$ on $M$ and $C_0>0$ such that$$
       d_{T}(p,x)+1\leq\gamma(x)\leq  d_{T}(p,x)+C_0$$ and $$
|\nabla^k_T \gamma| \le C_0
$$for $1\le k\le 2$, where $d_{T}$ is the distance function and $\nabla_T$ is the covariant derivative   with respect to $g(T)$. Let $f:(M,g_0)\to (N,h)$ be a smooth map such that
$$
\sup_M e(f)\le e_0.
$$
 Then there exists a smooth solution $F$ to the heat flow for harmonic map with initial map $f$ defined on $M\times[0,T_0]$ such that
$$
\sup_{M\times [0,T_0]}e(F)<\infty
$$
for some $0<T_0\le T$ depending only on $m, n, T, K_0, \kappa, L,  e_0, C_0$.
\end{prop}

Let us  prove Theorem \ref{t-existence-main} assuming the proposition is true.

\begin{proof}
Fix any small $T>s>0$. Let $ g^{(s)} (t):=g(t+s), 0\leq t\leq T-s$. Then
$$
\ppt g^{(s)}(t)=H(s+t)=:H^{(s)}(t).
$$
By condition  {\bf(a2)}, there is a constant $C_1=C_1(a,s)$ such that $|H^{(s)} (t)|, |\n H^{(s)}(t)|$ are uniformly bounded by $C_1$ on $M\times[0,T-s]$.  Here the norm and covariant derivative are with respect to $g^{s}(t)$. By conditions {\bf (a1)}, we have
$$
2\Ric(g(t))\ge -(K(t)+C_1)g(t).
$$
Together with {\bf (a3)} and the condition on $\gamma$, by Proposition \ref{p-existence-2}, for any smooth map $\wt f:M\to N$ with energy density bounded by $\wt e_0$, and for any $T-s>t_0>0$,  the harmonic map heat flow has a solution $\wt F$ in $M\times[t_0,t_0+\wt T_0]$ with initial map $\wt f$ so that
$$
\sup_{M\times[t_0,t_0+\wt T_0]}e(\wt F)<\infty.
$$
Here $\wt T_0$ depends only on $m, n, T-s, K_0, \kappa, \wt e_0, C_1, C_0$ as long as $t_0+\wt T_0\le T-s$.

 In particular,  the harmonic map heat flow has a solution $ F^{(s)}$ in $M\times[0,\wt T_1]$ with initial map $  f$ so that
$$
\sup_{M\times[0,\wt T_1]}e(F^{(s)})<\infty.
$$
Here $ \wt T_1$ depends only on $m, n, T-s, K_0, C_1, \kappa,   e_0, C_0$ as long as $ \wt T_1\le T-s$. By Lemma \ref{l-e-energy}, we conclude that
$$
\sup_{M\times[0,\wt T_1]}e( F^{(s)})\le 2e_0\exp(K_0)
$$
provided $\wt T_1\le \frac12(2\kappa e_0\exp(K_0))^{-1}$. If this is the case, then one can extend the solution to $[0, \wt T_1+\wt T_0]$ by Lemma \ref{l-extension}, provided $\wt T_1+\wt T_0\le T-s$, where $\wt T_0$ depends only on $m, n, T-s, K_0, C_1,  \kappa, C_0$
and
$$
\wt e_0:=2e_0\exp(K_0).
$$
Continue in this way, we conclude that the harmonic map heat flow has a solution $ F^{(s)}$ in $M\times[0,T_s]$ with initial map $  f$ so that
$$
\sup_{M\times[0,  T_s]}e(F^{(s)})\le 2e_0\exp(K_0)
$$
where
$$
T_s= \min\{T-s, \frac12\lf( 2\kappa e_0\exp(K_0) \ri)^{-1}\}.
$$
By Lemma \ref{l-e-tension}, we have
$$
|\tau(F^{(s)})|_{g^{(s)}(t)}\le C_2t^{-\frac12}
$$
where $C_2$ depends only on $m, n, e_0, a, K_0, \kappa$.

Hence   $d_N(F^{(s)}(x,t),f(x))\le C_3$ in $M\times[0,T_s]$ for some $C_3$ independent of $s$. In local coordinates $F^{(s)}$ satisfies a system of semi-linear equations

\bee
\ppt (F^{(s)})^\a(x, t)=g_s^{ij}(x, t)((F^{(s)})^\a_{ij}-\Gamma^k_{ij}(g_s(x, t))(F^{(s)})^\a_k+\wt\Gamma^\a_{\b\gamma}(F^{(s)}) ^\b_i(F^{(s)})^\gamma_j).
\eee
Moreover, $|\nabla (F^{(s)})^\a|$ are uniformly bounded. Then we can argue as in the proof of Lemma \ref{l-semi} to conclude that for any precompact domain $\Omega\subset M$, all orders of derivatives of $ F^{(s)} $ are uniformly bounded in $\Omega\times[0,T_s]$. Passing to a subsequence, we conclude that $ F^{(s)} $ will converge on $M\times[0,T_0]$ to a solution of the harmonic map heat flow coupled with $g(t)$ on $M\times[0,T_0]$ with initial map being $f$ such that $e(F)$ and $|\tau(F)|$ have bounds as stated in the theorem.

\end{proof}

\subsection{Proof of Proposition \ref{p-existence-2}}

Our method is to use conformal change to find solutions on compact domains. We then obtain estimates for the energy density and the norm of the tension field in order to take limit as in the proof of Theorem \ref{t-existence-main}.
Let $\chi\in (0,\frac{1}{8})$, $f:[0,1)\to[0,\infty)$ be the function:
\be\label{e-exh-1}
 f(s)=\left\{
  \begin{array}{ll}
    0, & \hbox{$s\in[0,1-\chi]$;} \\
    -\displaystyle{\log \lf[1-\lf(\frac{ s-1+\chi}{\chi}\ri)^2\ri]}, & \hbox{$s\in (1-\chi,1)$.}
  \end{array}
\right.
\ee

Let   $\varphi\ge0$ be a smooth function on $\R$ such that $\varphi(s)=0$ if $s\le 1-\chi+\chi^2 $, $\varphi(s)=1$ for $s\ge 1-\chi+2 \chi^2 $
\be\label{e-exh-2}
 \varphi(s)=\left\{
  \begin{array}{ll}
    0, & \hbox{$s\in[0,1-\chi+\chi^2]$;} \\
    1, & \hbox{$s\in (1-\chi+2\chi^2,1)$.}
  \end{array}
\right.
\ee
such that $\displaystyle{\frac2{\chi^2}}\ge\varphi'\ge0$. Define
 $$\mathfrak{F}(s):=\int_0^s\varphi(\tau)f'(\tau)d\tau.$$

From \cite{Lee-Tam-2017-2}, we have:

\begin{lma}\label{l-exhaustion-1} Suppose   $0<\chi<\frac18$. Then the function $\mathfrak{F}\ge0$ defined above is smooth and satisfies the following:
\begin{enumerate}
  \item [(i)] $\mathfrak{F}(s)=0$ for $0\le s\le 1-\chi+\chi^2$.
  \item [(ii)] $\mathfrak{F}'\ge0$ and for any $k\ge 1$, $\exp( -k\mathfrak{F})\mathfrak{F}^{(k)}$ is uniformly  bounded.
  %
\end{enumerate}

\end{lma}
Let $\gamma$ be the exhaustion function as in the assumption of the proposition. Let $\chi=\frac{1}{16}$.
For $\rho>1$, let $U_\rho$ be the component of $\gamma^{-1}([0,\rho))$ containing a fixed point $p$. Note that $U_\rho$ exhausts $M$ as $\rho\to\infty$.  Now we consider a function on $U_\rho$ defined by \bee
\phi(x):=\mathfrak{F}(\frac{\gamma(x)}{\rho}).
\eee
and let
$$
\wt g(t):=\exp(2\phi) g(t).
$$
Then $\wt g$ is a smooth family of complete metrics on $U_\rho$ so that
$$
\ppt \wt g=\wt H
$$
where $\wt H=\exp(2\phi)H$.
\begin{lma}\label{l-conformal-3} With the above notations, under the assumptions as in Proposition \ref{p-existence-2}, we have in $U_\rho\times[0,T]$:
\begin{enumerate}
  \item [(i)] $|\wt H(t)|_{\wt g(t)}, |\tn \wt H(t)|_{\wt g(t)}\le C$ for some constant $C$ depending only on $L, T, C_0$, where $\tn$ is the derivative with respect to $\wt g(t)$.
  \item [(ii)] $2\Ric(\t g(t))+\t H(t)\ge -\t K(t)\t g(t)$ for some $\t K\ge0$ so that $\int_0^T\t K(t)dt\le \t K_0$ for some constant $\t K_0$ depending only on $L, T, K_0, C_0$.
  \item [(iii)] $|\Rm(\t g(t))|$ is uniformly bounded.
\end{enumerate}

\end{lma}
\begin{proof} Since $|H|\le L$, we have
\be\label{e-conformal-1}
C_1^{-1}g(t)\le g(T)\le C_1 g(t)
\ee
for some $C_1=C_1(L,T)$. On the other hand, since $|\n H|\le L$ and $\ppt g=H$, if we let $\Gamma$ and $\bar \Gamma$ be the Christoffel symbols of $g(t)$ and $\bar g=g(T)$ respectively and let $A=\Gamma-\bar\Gamma$, we have
$|A|_{g(t)}$ is bounded by a constant depending only on $L, T$.
Since

\bee
^t\nabla^2\gamma=\nabla_T^2\gamma+A*\nabla_T\gamma,
\eee we have
\be\label{e-exhaustion}
| \nabla \gamma|_{g(t)}\le C_2,\ \  | \nabla^2\gamma|_{g(t)}\le C_2,
\ee
for some constant $C_2=C_2(L,T,C_0)$.  Next we want to compute the gradient and Hessian of $\phi$. Let the covariant derivative with respect to $g(t)$ be denoted by ${;}$, then
$$
 \phi_i=\rho^{-1}\mathfrak{F}' \gamma_i,
$$
$$
\phi_{;ij}=\rho^{-1}\mathfrak{F}' \gamma_{;ij}+\rho^{-2}\mathfrak{F}''\gamma_i\gamma_j.
$$
By Lemma \ref{l-exhaustion-1},
\be\label{e-conformal-2}
\left\{
  \begin{array}{ll}
    |\n \phi|_{g(t)}\le C_3\exp(\phi);  \\
   |\n ^2\phi|_{g(t)}\le C_3 \exp(2\phi).
  \end{array}
\right.
\ee
for some $C_3=C_3(L,T,C_0)$ because $\phi\ge0$.

\bee
\begin{split}
|\wt H|^2_{\wt g(t)}=&\wt g^{ij}\wt g^{kl}\wt H_{ik}\wt H_{jl}\\
=&e^{-4\phi}g^{ij}g^{kl} e^{4\phi}H_{ik}H_{ij}\\
=&|H|^2_{g(t)}.
\end{split}
\eee

Since
\bee\begin{split}
\t\n\t H=&(\t\n-\n)\t H+\n\t H\\
=&(\t\Gamma-\Gamma)\ast e^{2\phi}\ast H+2e^{2\phi}\phi'\rho^{-1}\n\gamma\ast H\\
&+e^{2\phi}\ast H\\
=&(2\phi'\rho^{-1}\n\gamma\ast g\ast g^{-1})\ast e^{2\phi}\ast H\\
&+2e^{2\phi}\phi'\rho^{-1}\n\gamma\ast H+e^{2\phi}\ast H \end{split},\eee
we have
\be\label{e-wtH}
  |\t\n\t H|_{\t g}\le C_4
\ee
for some $C_4=C_4(L,T,C_0)$. These prove (i).

To prove (ii), by \eqref{e-conformal-2}, denote the Ricci tensor of $\t g$ by $\t R_{ij}$ and the Ricci tensor of $g(t)$ by $R_{ij}$, then
\bee
\begin{split}
 \t R_{ij} =&R_{ij}+(m-2)\phi_{;ij}+(m-2)\phi_i\phi_j-[\Delta_{g(t)} \phi+(m-2)|\n \phi|^2]g_{ij}\\
\ge & -K(t)g_{ij} -C_5\exp(2\phi)g_{ij}\\
\ge&-\t K(t) \t g_{ij}
\end{split}
\eee
where $\t K(t)=K(t)+C_5(L,T,C_0)$, because $\phi\ge0$.

To prove (iii),

     $$|\wt \Rm|_{\wt g}\le C_6\exp(-2\phi)\lf(|\Rm|_{g}+ |\nabla^2 \phi|_g+|\nabla \phi|^2_g\ri)
     $$
     for some $C_6=C_6(m) $. By using \eqref{e-conformal-2} and the fact that the curvature of $g$ is uniformly bounded in $U_\rho\times[0,T]$ the result follows.

\end{proof}

We are now ready to prove Proposition \ref{p-existence-2}.

\begin{proof}[Proof of Proposition \ref{p-existence-2}] Let $\phi, \rho, U_\rho, \t g$ be as above. We claim that there is $T_0>0$ depending only on $m, n, K_0, \kappa, e_0, L, C_0, T$ such that the heat flow for harmonic map from $(U_\rho,\t g(t))$ to $N$ with initial map $f|_{U_\rho}$ has a solution $F^{(\rho)}$ on $U_\rho\times[0,T_0]$ so that
\bee
\sup_{U_\rho\times[0,T_0]}e(F^{(\rho)})<\infty.
\eee
Suppose the claim is true, using the fact that $2\Ric(\t g(t))+\wt H(t)\ge -\t K(t) \t g(t)$ for some $\t K$ with
$$
\t K_0:=\int_0^T \t K<\infty
$$
where $\t K_0$ depends only on $K_0,   L, C_0, T$, one can proceed as in the proof of Theorem \ref{t-existence-main} to conclude the proposition is true by taking limit of a subsequence of $F^{(\rho)}$ with $\rho\to\infty$.

In order to prove the claim, we use the method as in \cite{E-S} and \cite{Li-Tam}.  Since the image of $f|_{U_\rho}$ is bounded in $N$. Let $f|_{U_\rho}\Subset \Omega\Subset N$.
Here $\Omega$ is a bounded domain in $N$ and let $\Omega_1$ be another bounded domain with $\Omega\Subset \Omega_1$. Isometrically embed a neighborhood $O$ of  $\Omega_1$ into $\mathbb{R}^q$ for some $q\in\mathbb{N}$. Let $W$ be a bounded tubular neighborhood of $O$ in  $\mathbb{R}^q$. Let $\pi:W\to O$ be the nearest point projection. Write $$
\pi=(\pi^1, \pi^2, \cdots, \pi^q)=(\pi^A)_{1\leq A\leq q}. $$
We can extend $\pi$ smoothly from a possible smaller tubular neighborhood $V$ of $\Omega_1$ to the whole $\mathbb{R}^q$ such that each $\pi^A$ is compactly supported and  $\pi$ is not changed in $V$. Hence, $\pi^A$, $\pi^A_B:=\frac{\p\pi^A}{\p z^B}$, $\pi^A_{BC}:=\frac{\p\pi^A}{\p z^B\p z^C}$ etc are bounded, where $z=(z^A)$ are the standard coordinates of $\mathbb{R}^q$.

Let $f:U_{\rho}\to N$ so that $f(U_{\rho})\Subset \Omega$. Then we can write
$$
f(x)=(f^A(x))\in\R^q.
$$
Note that $e(f)=\sum\limits_A|\nabla f^A(x)|^2$. By Lemma \ref{l-conformal-3},
$$
|\t H|_{\t g(t)}, |\tn\t H|_{\t g(t)}\le C_1
$$
for some $C_1=C_1(L,T,C_0)$. $|\Rm(\t g(t))|\le Q$ which may also depend on $\rho$.
We may assume
$$
|\pi^A_{BC}|\le D.
$$
Consider the following system of equations:
\be\label{harmonic-pde}
\heat F^A=-\pi^A_{BC}(F)\la \nabla F^B,\nabla F^C\ra  \ee in $U_{\rho}\times(0, T]$  and $F^A(x, 0)=f^A(x)$ in $U_{\rho}$ for $A=1, 2, \cdots, q$. By Lemma \ref{l-semi}, the system has a smooth solution on $U_{\rho}\times[0,T_1]$, where
$ T_1$ depends only on  $m, q, D,  L, T,  C_0, Q, \sup_{U_\rho}e(f)$.  Moreover, $F$ and $|\nabla F|$ are uniformly bounded by a constant $C$ depending only on  $m, q, D,  L, T,  C_0, Q, \sup_{U_\rho}e(f)$ and $\Omega$.
On the other hand,
$$
F^A(x,t)= \int_M G(x,t;y,0)f^A(y)dy+\int_0^t\int_MG(x,t;y,s)Q^A(y,s) dV_s(y) ds
$$
where $G$ is the fundamental solution to the heat equation coupled with $g(t)$ and
$Q^A$ is the RHS of \eqref{harmonic-pde}. By Proposition \ref{p-heat-grad},  we conclude that
$$
|F^A(x,t)-f^A(x)|\le C_2t^\frac12
$$
for some constant $C_2=C_2(m, q, D, Q, T, \sup_{U_\rho}e(f))$. In particular, there exists $0<T_2\leq T_1$ depending only on  $m, q, D, Q, T, \sup_{U_\rho}e(f)$ such that $F(x,t)$ will be inside the tubular neighborhood $W$ of $O$. By the proof of \cite[Lemma 3.2]{Li-Tam} and the maximum principle Theorem \ref{max}, we conclude $F(x,t)\in N$ for $(x,t)\in U_\rho\times[0,T_2]$. This implies $F(x,t)$ satisfies the harmonic heat flow on $U_\rho\times[0,T_2]$ to $N$.

Up to now we have proved the following: If $f:U_\rho\to N$ is a  smooth bounded map with energy density bounded, then there is a smooth solution $F$ to the heat flow for harmonic map with initial data $f$ on $U_\rho\times[0,T_2]$  so that the image and the energy density of $F$ are uniformly bounded. By Lemma \ref{l-extension}, we conclude that as long as the energy of $F$ is uniformly bounded on $U_\rho\times[0,T']$ and $F$ has bounded image, then $F$ can be extended beyond $T'$ as a solution to the harmonic map heat flow so that the energy is uniformly bounded.

 Recall that $2\Ric(\t g)+\t H\ge -\t K(t)\t g(t)$ with $\t K_0=\int_0^T \t K<\infty$. Using this condition, by Lemma \ref{l-e-energy}, we conclude
that
\bee
e(F)(\cdot,t)\le 2e_0\exp(\t K_0)
\eee
as long as $0\le t\leq T_0:=\min\{T, \frac 12\lf( 2\kappa e_0 \exp(\t K_0)\ri)^{-1}\}$. By Lemma \ref{l-e-tension}, there is constant $C_3$ depending only on $m, n, e_0, L,\t K_0, T, \kappa$ such that
$$
|\tau(F)|(x,t)\le C_3t^{-\frac12}
$$
as long as $0\le t\le T_0$. Since $\tau(F)=F_t$, we have $d_N(f(x), F(x,t))\le C_5t^\frac12$. In particular, the image of $U_\rho\times[0,T_0]$ is bounded because $f(U_\rho)$ is bounded. This completes the proof of the claim and hence the proposition.
\end{proof}

\end{document}